\documentclass[11pt]{amsart}

\usepackage{amsthm}
\usepackage{amsmath}
\usepackage{amssymb}

\newtheorem{thm}{Theorem}[section]

\newtheorem*{paving conjecture}{Paving Conjecture}
\newtheorem*{bourgain-tzafriri}{Bourgain-Tzafriri Conjecture}

\newtheorem{lemma}[thm]{Lemma}
\newtheorem{proposition}[thm]{Proposition}

\newtheorem{definition}[thm]{Definition}
\theoremstyle{remark}

\newcommand{\NN}{\mathbb N}

\newcommand{\cH}{\mathcal H}

\title{concrete constructions of non-pavable projections}

\author[Casazza, Fickus, Mixon, Tremain]{Peter G. Casazza, Matt Fickus,
Dustin Mixon and Janet C. Tremain}
\date{\small\today}
\address{Casazza/Tremain:  Department of Mathematics, University
of Missouri, Columbia, MO 65211-4100\\
Fickus/MixonDepartment of Mathematics and Statistics, Air Force Institute of Technology,
Wright-Patterson Air Force Base, Ohio 45433\\
}

\email{casazzap@missouri.edu, Matthew.Fickus@afit.edu,j.tremain@mchsi.com}

\begin{document}

\begin{abstract}
It is known that the paving conjecture fails for $2$-paving projections with constant
diagonal $1/2$.  But the proofs of this fact are existence proofs.  We will give 
concrete examples of these projections and projections with constant diagonal
$1/r$ which are not $r$-pavable in a very strong sense. \\


\end{abstract}

\maketitle

\section{Introduction}

It is now known that the 1959 Kadison-Singer Problem is equivalent to fundamental
unsolved problems in a dozen areas of research in pure mathematics, applied
mathematics and engineering \cite{CT,CT1}.  In 1979, Anderson \cite{A} showed
that the Kadison-Singer Problem is equivalent to the {\em Paving Conjecture}.

\begin{paving conjecture}[PC]
For $\epsilon >0$, there is a natural number $r$ so
that for every
 natural number $n$ and
every linear operator
  $T$ on $l_2^n$
whose matrix has zero diagonal,
  we can find a partition (i.e. a {\it paving})
$\{{A}_j\}_{j=1}^r$
  of $\{1, \ldots, n\}$, such that
  $$
  \|Q_{{A}_j} T Q_{{A}_j}\|  \le  \epsilon \|T\|
  \ \ \ \text{for all $j=1,2,\ldots ,r$,}
  $$
  where $Q_{A_j}$ is the natural projection onto the $A_j$ coordinates of a vector.
\end{paving conjecture}

Operators satisfying the Paving Conjecture are called {\bf pavable operators}.
A projection $P$ on $\cH_n$ is {\bf $(\epsilon,r)$-pavable} if there
is a partition $\{A_j\}_{j=1}^r$ of $\{1,2,\ldots,n\}$ satisfying
\[ \|Q_{A_j}PQ_{A_j}\| \le \epsilon,\ \ \mbox{for all $j=1,2,\ldots,r$}.\]

It was shown in \cite{CE} that  projections with
constant diagonal $1/r$ are not $(r,\epsilon)$-pavable for any $\epsilon >0$.
But the argument in \cite{CE} is an existence proof and the actual matrices failing
paving were not known.  In this note we will construct concrete examples of
these projections.  As a consequence, we will obtain a stronger result than that
of \cite{CE}.

\section{Preliminaries}

We will actually work with an equivalent form of the Paving Conjecture for
projections with constant diagonal.
In 1989, Bourgain and Tzafriri proved one of the most celebrated theorems in analysis:
The {\em Bourgain-Tzafriri Restricted Invertibility Theorem} \cite{BT}.  This gave rise to
a major open problem in analysis.

\begin{bourgain-tzafriri}[BT]\label{CBST1}
There is a universal constant $A>0$ so that
for every $B>1$ there is a natural number $r=r(B)$
satisfying:
For any natural number $n$,
if $T:{\ell}_2^n \rightarrow {\ell}_2^n$ is a linear operator
with $\|T\|\le B$ and $\|Te_{i}\|=1$ for all $i=1,2,\ldots , n$,
 then there is a partition
$\{A_{j}\}_{j=1}^{r}$ of $\{1,2,\ldots , n\}$ so that
for all $j=1,2,\ldots ,r$ and
 all choices of scalars $\{a_{i}\}_{i\in A_{j}}$ we have:
$$
\|\sum_{i\in A_{j}}a_{i}Te_{i}\|^{2}\ge A \sum_{i\in A_{j}}|a_{i}|^{2}.
$$
\end{bourgain-tzafriri}

It was shown in \cite{CT} that BT is equivalent to the Paving Conjecture.    

\begin{definition}
A family of vectors $\{f_i\}_{i=1}^M$ for an $n$-dimensional Hilbert space $\cH_n$ is
$(\delta,r)$-{\bf Rieszable} if there is a partition $\{A_j\}_{j=1}^r$ of
$\{1,2,\ldots,M\}$ so that for all $j=1,2,\ldots,r$ and all scalars $\{a_i\}_{i\in A_j}$
we have
\[ \|\sum_{i\in A_j}a_if_i\|^2 \ge \delta \sum_{i\in A_j}|a_i|^2.\]
A projection $P$ on $\cH_n$ is $(\delta,r)$-{\bf Rieszable} if $\{Pe_i\}_{i=1}^n$
is $(\delta,r)$-Rieszable.
\end{definition}

Recall that a family of vectors $\{f_i\}_{i\in I}$ is a {\bf frame} for a Hilbert space 
$\cH$ if there are constants $0<A,B < \infty$, called the {\bf lower  (upper)
frame bounds)} respectively satisfying for all $f\in \cH$:
\[ A \|f\|^2 \le \sum_{i\in I}|\langle f,f_i\rangle |^2 \le B\|f\|^2.\]
If $\|f_i\|=\|f_j\|$ for all $i,j$, we call this an {\bf equal norm} frame and if
$\|f_i\|=1$ for all $i$, it is a {\bf unit norm frame}.
If $A=B$ this is an {\bf $A$-tight frame} and if $A=B=1$, it is a {\bf Parseval frame}.
It is known \cite{C,CK,Ch} that $\{f_i\}_{i\in I}$ is an $A$-tight frame if and only if
the matrix with the $f_i's$ as rows has orthogonal columns and the
square sums of the column coefficients equal $A$.  It is also known 
\cite{C,Ch} that $\{f_i\}_{i=1}^M$ is a Parseval frame for $\cH_n$ if and only if there is an
othogonal projection $P:\ell_2^M \rightarrow \cH_n$ with 
\[ Pe_i = f_i,\ \ \mbox{for all $i=1,2,\ldots,M$},\]
where $\{e_i\}_{i=1}^M$ is the unit vector basis of $\ell_2^M$.

The following result  can be found in \cite{CE,T}.  

\begin{proposition}\label{prop1}
Fix a natural number $r\in \NN$.  The following are equivalent:

(1)  The Paving Conjecture.

(2)  The class of projections with constant diagonal $1/r$ are pavable.

(3)  The class of projections with constant diagonal $1/r$ are Rieszable.

(4)  The class of unit norm $r$-tight frames $\{f_m\}_{m=1}^{nr}$ for $\cH_n$ are
Rieszable.
\end{proposition}

We will construct concrete counterexamples for (4) of \ref{prop1}.  These will give
concrete counterexamples to 1-4 in the proposition by the following result which
can be found in \cite{CE}.  The point here is that the proof of this proposition gives
an explicit representation of each of the equivalences in the proposition in terms
of all the others.

\begin{proposition}
Let $P$ be an orthogonal projection on $\cH_n$ with matrix $B = (a_{ij})_{i,j=1}^n$.
The following are equivalent:

(1)  The vectors $\{Pe_i\}_{i=1}^n$ is $(\delta,r)$-Rieszable.

(2)  There is a partition $\{A_j\}_{j=1}^r$ of $\{1,2,\ldots,n\}$ so that
for all $j=1,2,\ldots,r$ and all scalars $\{a_i\}_{i\in A_j}$ we have
\[ \|\sum_{i\in A_j}a_i(I-P)e_i\|^2 \le (1-\delta)\sum_{i\in I}|a_i|^2.\]

(3)  The matrix of $I-P$ is $(\delta,r)$-pavable.        
\end{proposition} 

As a fundamental tool in our work, we will work with the $n\times n$
discrete Fourier transform matrices which we will just call DFT matrices
or $DFT_{n\times n}$.   For these, we fix $n\in \NN$ and let $\omega$ be a 
primative $n^{th}$ root of unity and define
\[ DFT_{n\times n} = \left ( \frac{1}{\sqrt{n}}\omega^{ij}\right )_{i,j=1}^n.\]

The main point of these $DFT_{n\times n}$ 
matrices is that they are unitary matrices for which the modulus of all of
the entries of the matrix are equal to 1.  We will use on the following simple observation.

\begin{proposition}\label{prop5}
If $A=(a_{ij}\}_{i,j=1}^n$ is a matrix with $|a_{ij}|^2=a$ for all $i,j$ and
orthogonal columns and we multiply the
$j^{th}$-column of $A$ by a constant $C_j$ to get a new matrix $B$, then 

(1)  The columns of $B$ are orthogonal.

(2)  The square sums of the coefficients of any row of $B$ all equal
\[a \sum_{j=1}^n C_j^2.\]

(3)  The square sum of the coefficients of the $j^{th}$ column of $B$ equal
$aC_j^2$.
\end{proposition}

\section{The Case $r$=2}

Let us first outline our construction.
For any natural number $n$, we will alter two $2n \times 2n$ DFT matrices 
along the lines of Proposition \ref{prop5} and then stack them on top
of one another to get a $4n\times 2n$ matrix with the following
properties:

(1)   Each altered DFT has the square sums of the
coefficients of any row equal to 1.

(2)  The top altered DFT will have the square sums of the coefficients of each
column $j$ with $1\le j\le n-1$ equal to 2, and the square sums of the
coefficients of the remaining columns will all equal $2/(n+1)$.

(3)  The combined matrix will have the square sums of the coefficients of each
column equal to 2.

(4)  The columns of the combined matrix are orthogonal.

It follows that this is the matrix of a unit norm $2$-tight frame and hence multiplying the matrix by
$1/\sqrt{2}$ will turn it into an equal norm Parseval frame, creating the matrix of a 
rank $2n$ projection
on ${\mathcal C}^{4n}$ with constant diagonal $1/2$.  We will then show that the rows
of this class of matrices are not uniformly $2$-Rieszable to complete the example.

So we start with a $2n\times 2n$ DFT and multiply the first $n-1$ columns
by $\sqrt{2}$ and the remaining columns by $\sqrt{\frac{2}{n+1}}$
to get a new matrix $B_1$. 
Now, we take the second $2n\times 2n$ DFT matrix and multiply the first 
$n-1$ columns by $0$ and the remaining columns by $\sqrt{\frac{2n}{n+1}}$ to get a
matrix $B_2$.  We form the matrix $B$ by stacking the matrices $B_1$ and $B_2$
on top of one another to get the matrix $B$ given below.

\begin{center}
\begin{tabular}{|c|c|}
\hline (n-1)-Colmns & (n+1)-Colmns.\\
\hline $\sqrt{2}$ & $\sqrt{\frac{2}{n+1}}$  \\ 
\hline 0 & $\sqrt{\frac{2n}{n+1}}$  \\ 
\hline 
\end{tabular} 
\end{center}
\vskip12pt

Now we can prove:

\begin{proposition}
The matrix $B$ satisfies:

(1)  The columns are orthogonal and the square sum of the coefficients of
every column equals 2.

(2)  The square sum of the coefficients of every row equals 1.

The row vectors of the matrix $B$ are not $(\delta,2)$-Rieszable, for any
$\delta$ independent of $n$.
\end{proposition}

\begin{proof}
Clearly the columns of $B$ are orthogonal.  To check the square sums of the
column coefficients, recall that for columns $1\le \ell \le n-1$ the modulus
of all the coefficients of $B_1$ are $\frac{1}{\sqrt{n}}$, the the coefficients of $B_2$
are 0.  So the square sum of the coefficients in column $\ell$ are:
\[ \frac{1}{n}\cdot 2n + 0 = 2.\]
For the columns $n\le \ell \le 2n$, the modulus of the coefficients of $B_1$ are
$\frac{1}{\sqrt{n(n+1)}}$ and the coefficients of $B_2$ are $\frac{1}{\sqrt{n+1}}$.
So the square sum of the coefficients of $B$ in column $\ell$ are:
\[ 2n \cdot \frac{1}{n(n+1)} + 2n \cdot \frac{1}{n+1} = \frac{2}{n+1} + \frac{2n}{n+1} = 2.\]

Now we check the row sums.  For any row of $B_1$, the first $n-1$ column coefficients
have modulus $\frac{1}{\sqrt{n}}$, and the modulus of the coefficients of the last
$n+1$ columns of $B_1$ have modulus $\frac{1}{\sqrt{n(n+1)}}$.  So the square sum of the
coefficients of any row of $B_1$ are:
\[ (n-1)\frac{1}{n} + (n+1)\frac{1}{n(n+1)} =1.\]
For any row of $B_2$, the first $n-1$ column coefficients are equal to 0 and the
remaining $n+1$ column coefficients have modulus $\frac{1}{\sqrt{n+1}}$.  So the
square sum of the row coefficients of $B_2$ are
\[ (n+1)\frac{1}{n+1} + 0 = 1. \]

We will now show that the row vectors of $B$ are not two Rieszable.
So let $\{A_1,A_2\}$ be a partition of $\{1,2,\ldots,4n\}$.  Without loss of generality,
we may assume that $|A_1 \cap \{1,2,\ldots,2n\}| \ge n$.    Let the row vectors of the matrix $B$
be $\{f_i\}_{i=1}^{4n}$ as elements of ${\mathcal C}^{2n}$.  Let $P_{n-1}$ be
the orthogonal projection of ${\mathcal C}^{2n}$ onto the first $n-1$ coordinates.
Since $|A_1|\ge n$, there are scalars $\{a_i\}_{i\in A_1}$ 
so that $\sum_{i\in A_1}|a_i|^2 = 1$ and 
\[ P_{n-1}\left ( \sum_{i\in A_1}a_if_i \right ) = 0.\]  Also, let $\{g_j\}_{j=1}^{2n}$ be
the orthonormal basis consisting of the original rows of the $DFT_{2n\times 2n}$.
We now have:    
\begin{eqnarray*}
\| \sum_{i\in A_1}a_if_i\|^2 &=&\|(I-P_{n-1})\left ( \sum_{i\in A_1}a_if_i \right )\|^2\\
&=& \frac{2}{n+1}\|(I-P_{n-1})\left ( \sum_{i\in A_1}a_ig_i\right ) \|^2\\
&\le& \frac{2}{n+1}\|\sum_{i\in A_1}a_ig_i\|^2\\
&=& \frac{2}{n+1}\sum_{i\in A_1}|a_i|^2\\
&=& \frac{2}{n+1}.
\end{eqnarray*}
Letting $n\rightarrow \infty$, we have that this class of matrices is not $(\delta,2)$-pavable
for any $\delta >0$.
\end{proof}

\section{The case of general $r$}

In this section we will extend our construction to projections with constant diagonal
$1/r$ and actually prove a stronger result.

\begin{proposition}\label{prop10}
For every natural number $r\ge 2$, there is a $r^2n \times rn$ projection matrix with constant
diagonal $1/r$ so that whenever we partition $\{1,2,\ldots,r^2n\}$ into sets $\{A_j\}_{j=1}^r$,
and for all $k=1,2,\ldots,r$, if $D_k = \{(k-1)rn+1,(k-1)rn+2,\ldots, krn\}$, 
then for every $k=1,2,\ldots,r-1$, there is a $j$ so that the vectors $\{f_i\}_{i\in A_j \cap D_k}$
are not uniformly $2$-Rieszable.    
\end{proposition}

This time, we will take $r$ DFT matrices of size $rn \times rn$ and alter their columns
by certain amounts so that when we stack them on top of one another 
to get a matrix $B$ of size $r^2n \times rn$ satifying:
\vskip10pt
1.  The columns of $B$ are orthogonal and the sums of the 
squares of the coefficients of each row of $B$ equals 1.
\vskip10pt
2.  The sums of the squares of the coefficients of each column of $B$ equals $r$.
\vskip10pt
3.  $B$ satisfies the requirements of the proposition.
\vskip10pt
For the first matrix $B_1$ we take the $rn\times rn$ DFT and multiply the first 
$n-1$ columns by $\sqrt{r}$
and the remaining columns by $\sqrt{\delta_1}$ (to be chosen later).
For $B_2$ we take the $rn\times rn$ DFT and multiply the first $n-1$ columns
by 0, multiply the columns $n-1+j$, $j=1,2,\ldots,n-1$ by $\sqrt{r - \delta_1}$,
and multiply the remaining columns by $\sqrt{\delta_2}$ (to be chosen later).
And for $k=3,\ldots, r-1$ we construct the matrix $B_k$ by 
taking the $rn \times rn$ DFT and multiplying  the first
$(k-1)(n-1)$ columns by 0, multiply the columns $(k-1)(n-1)+j$ for $j=1,2,\ldots,n-1$ by
\[ \sqrt{r - \sum_{i=1}^{k-1}\delta_{k-1}},\]
 and multiplying the remaining columns by $\sqrt{\delta_k}$ (to be chosen later).
 Finally, for $B_r$ we take the $rn\times rn$ DFT and multiplying the first 
 $(r-1)(n-1)$ columns by 0 and the remaining columns by $\sqrt{\delta_r}$
 (to be chosen later).

We then stack these $r$, $rn\times rn$ matrices $\{B_k\}_{k=1}^r$ on top of each other
to produce the matrix $B$ for which the moduli
of the coefficients of $B$ are given in figure 2 below.  Now we must show that the matrix
$B$ has all of the properties of Proposition \ref{prop10}.  

It is clear that the columns of $B$ are orthogonal.
To show that the square sums of the row coefficients of the matrix $B$ are all equal to 1,
we need a lemma.

\begin{lemma}\label{lem1}
To get the rows of the matrix $B$ to square sum to 1,
we need
\begin{equation}\label{E3}
 \delta_k = \frac{r^2n}{[(r-k+1)n+k-1][(r-k)n+k]}.
 \end{equation} 
\end{lemma}

\begin{proof}
We will proceed by induction on $k$ to show Equation \ref{E3} for
all $k=1,2,\ldots, r$.  For $k=1$, we observe that the coefficients 
of the first $n-1$ columns of $B_1$ have modulus equal to $1/n$,
while the coefficients of the remaining $rn-(n-1)$ columns of $B_1$ have modulus
$\sqrt{\frac{\delta_1}{rn}}$.  So the sum of the squares of the coefficients of any
row of $B_1$ equals
\[ \frac{1}{rn}\left [ r(n-1)+\delta_1(rn-(n-1)) \right ] = 1.\]
Hence,
\[ \delta_1(rn-(n-1)) = rn -r(n-1) = r.\]
So,
\[ \delta_1 = \frac{r}{(r-1)n+1}= \frac{r^2n}{[(r-1+1)n+1-1][(r-1)n+1]|}.\]
For $k=2$, our matrix $B_2$ has coefficients of the first $n-1$ columns equal to 0,
coefficients of the columns $(n-1)+j$, $j=1,2,\ldots,n-1$ have modulus equal to
\[ \sqrt{\frac{r-\delta_1}{rn}}.\]
 and the remaining $rn-2(n-1)$ columns have modulus equal to $\sqrt {\frac{\delta_2}{rn}}$.
So the square sums of the coefficients of any row of $B_2$ equals
\[ \frac{1}{rn}\left [ (n-1)(r-\delta_1)+ (rn-2(n-1))\delta_2\right ] = 1.\]
Since
\[ r-\delta_1 = r-\frac{r}{(r-1)(n+1)} = \frac{r(r-1)n}{(r-1)n+1},\]
we can solve the equation to get
\[ \delta_2 = \frac{r^2n}{[(r-1)n+1][(r-2)n+2]}.\]

Now assume our formula holds for any $k\le r-1$ and we check it for $k+1$.
The matrix $B_{k+1}$ has coefficients of the first $k(n-1)$ columns equal to 0,
coefficients of the columns $k(n-1)+j$, $j=1,2,\ldots,n-1$ of modulus
\[ \left ( \frac{r- \sum_{j=1}^k \delta_k }{rn}\right )^{1/2},\]
and the coefficients of the remaining columns have modulus $\sqrt{\frac{\delta_{k+1}}{rn}}$.
It follows that the square sums of the row coefficients of the matrix $B_{k+1}$ must
satisfy
\begin{equation}\label{E2} \left ( r-\sum_{j=1}^k \delta_j \right )(n-1) + \delta_{k+1} [ rn-(k+1)(n-1)] = rn.
\end{equation}
Hence, letting $a = rn/(n-1)$ we have
\begin{eqnarray*}
\sum_{j=1}^{k}\delta_j &=& r^2n\sum_{j=1}^k \frac{1}{[r-j+1)n + j-1][(r-j)n+j]}\\
&=& \frac{r^2n}{(n-1)^2} \sum_{j=1}^k \frac{1}{(a+1-j)(a-j)}\\
&=& \frac{r^2n}{(n-1)^2} \sum_{j=1}^k \left ( \frac{1}{a-j}-\frac{1}{a-(j-1)} \right )\\
&=& \frac{r^2n}{(n-1)^2} \frac{1}{a-k}-\frac{1}{a-0}\\
&=& \frac{r^2n}{(n-1)^2}\frac{k}{a(a-k)}\\
&=& \frac{r^2kn}{rn(rn-k(n-1))}\\
&=& \frac{rk}{(r-k)n+k}.
\end{eqnarray*}
Combining this with Equation \ref{E2} we have
\begin{eqnarray*}
\delta_{k+1} &=& \frac{r+(n-1)\sum_{j=1}^k \delta_j}{rn-(k+1)(n-1)}\\
&=& \frac{r+(n-1) \left ( \frac{rk}{(r-k)n+k}\right ) }{(r-k+1)n+k-1}\\
&=& \frac{r[(r-k)n+k] + (n-1)rk}{[(r-k+1)n+k-1][(r-k)n+k]}\\
&=& \frac{r^2n-rkn+rk+rnk-kr}{[(r-k+1)n+k-1][(r-k)n+k]}\\
&=& \frac{r^2n}{[(r-k+1)n+k-1][(r-k)n+k]}.
\end{eqnarray*}
\end{proof}

By Lemma \ref{lem1}, we know that the rows of the matrix $B$ square sum to 1.
Now we need to check the column sums.  Most of this is true by our definitions.
We check two cases:
\vskip12pt
\noindent {\bf Case 1}:
For a column $\ell = k(n-1)+j$, $k=1,2,\ldots,r-1$,  the column coefficients 
for $1\le j \le k-1$ and $i = jrn+m$, $m=1,2,\ldots,rn$, have
modulus $\sqrt{\frac{\delta_j}{rn}}$, and for $i=krn+m$, $m=1,2,\ldots,rm$ the modulus
of the coefficients are $\sqrt{\frac{r-\sum_{j=1}^{k-1}\delta_j}{rn}}$, and all other
coefficients are 0.  Hence, the square sum of the column coefficients is
\[ rn \sum_{j=1}^{k-1}\frac{\delta_j}{rn} + rn \left ( \frac{r-\sum_{j=1}^{k-1}\delta_j}{rn} \right )
= r.\]
\vskip12pt

\noindent {\bf Case 2}:  For a column $\ell = (r-1)(n-1)+j$, with $j=1,2,\ldots rn-(r-1)(n-1) =
r+n-1$, the square sum of the coefficients of column $\ell$ are (using our formula for the
sum of the $\delta_k$ above):
\[ \sum_{k=1}^r \delta_k = \frac{r^2}{(r-r)n+r} = r.
\]
\vskip12pt
Finally, we need to show that our matrix $B$ is not pavable (with paving
constants independent of $n$) in the strong sense given
in the proposition.  This follows similarly to the $DFT_{2n\times 2n}$ case.  Let $\{f_i\}_{i=1}^{r^2n}$
be the rows of the matrix $B$ and let $\{g_i\}_{i=1}^{rn}$ be the rows of the DFT
matrix.  Also, let $P_k$ be the orthogonal projection of ${\mathcal C}_2^{rn}$ onto
the first $k(n-1)$ coordinates.  
Now let $\{A_j\}_{j=1}^r$ be a partition of $\{1,2,\ldots,r^2n\}$ and
fix $1\le k \le r-1$.  Then there is a $j$ so that $|A_j \cap D_k|\ge n$.  Since
the vectors $\{f_i\}_{i\in A_j\cap D_k}$ have zero coordinates for all
$j=1,2,\ldots,(k-1)(n-1)$, and there are scalars $\{a_i\}_{i\in A_j\cap D_k}$ satisfying

1.  $\sum_{i\in A_j\cap D_k}|a_i|^2 =1$.

2.  We have
\[ P_{k}\left ( \sum_{i\in A_j\cap D_k}a_if_i \right ) =0.\]

It follows from our construction that
\begin{eqnarray*}
\|\sum_{i\in A_j\cap D_k}a_if_i\|^2 &=& \|(I-P_k)\left (\sum_{i\in A_j\cap D_k}a_if_i\right )\|^2\\
&=& \delta_k \|(I-P_k)\left ( \sum_{i\in A_j\cap D_k}a_ig_i\right )\|^2\\
&\le& \delta_k \|\sum_{i\in A_j \cap D_k}a_jg_j\|^2\\
&=& \delta_k \sum_{i\in A_j\cap D_k}|a_i|^2\\
&=& \delta_k
\end{eqnarray*}
Since $\lim_{n\rightarrow \infty}\delta_k =0$, it follows that our family of matrices are
not $2$-Rieszable in the strong sense of the Proposition.  This argument looks pictorially
as:
\vskip12pt

Each square is a $rn \times (n-1)$ submatrix
\vskip12pt
\begin{center}
\begin{tabular}{|c|c|c|c|c|}
\hline $\sqrt{\frac{r}{rn}}$ & $\sqrt{\frac{\delta_1}{rn}}$ & $\sqrt{\frac{\delta_1}{rn}}$ & $\cdots$ & $\sqrt{\frac{\delta_1}{rn}}$ \\ 
\hline 0 & $\sqrt{\frac{r}{rn}}$ & $\sqrt{\frac{\delta_2}{rn}}$ & $\cdots$ & $\sqrt{\frac{\delta_2}{rn}}$ \\ 
\hline 0 & 0 & $\sqrt{\frac{r}{rn}}$ & $\cdots$ & $\sqrt{\frac{\delta_3}{rn}}$ \\ 
\hline $\vdots$ & $\vdots$ & $\vdots$ & $\ddots$ & $\vdots$ \\ 
\hline 0 & 0 & 0 & $\cdots$ & $\sqrt{\frac{\delta_r}{rn}}$ \\ 
\hline 
\end{tabular} 
\end{center}

\vskip12pt
The main question is whether it is possible to take the concrete constructions in this
paper and generalize them to give a complete counterexample to the Paving Conjecture.

\vskip12pt
\noindent {\bf Acknowledgments}
\vskip10pt
Casazza and Tremain were supported by NSF DMS 0704216, Fickus was supported by
AFOSR FIATA 07337J001.  The views expressed in this article are those of the authors
and do not reflect the official policy or position of the United States Air Force, Department of
Defense, or the U.S. Government.


\begin{thebibliography}{WW}


\bibitem{A} J. Anderson, {\em Restrictions and representations
of states on $C^{*}$-algebras}, Trans. AMS {\bf 249} (1979)
303--329.

\bibitem{BT}  J. Bourgain and L. Tzafriri, {\em Invertibility
of ``large'' submatrices and applications to the geometry of
Banach spaces and Harmonic Analysis}, Israel J. Math.
{\bf 57} (1987) 137--224.

\bibitem{BT1}  J. Bourgain and L. Tzafriri, {\em On a problem
of Kadison and Singer}, J. Reine Angew. Math. {\bf 420} (1991), 1--43.

\bibitem{C}  P.G. Casazza, {\em The art of frame theory}, Taiwanese J. Math.
{\bf 4} No. 2 (2000) 129-201.

\bibitem{CE}  P.G. Casazza, D. Edidin, D. Kalra and V. Paulsen, {\em }
{\em Projections and the Kadison-Singer Problem}, Operators and Matrices,
{\bf 1} No. 3 (2007) 391--408.

\bibitem{CK} P.G. Casazza and J. Kova\v{c}evi\'c, {\em Equal norm
tight frames with erasures}, Adv. Comp. Math {\bf 18} (2003) 387--430.


\bibitem{CT}  P.G. Casazza and J.C. Tremain, {\em The Kadison-Singer Problem
in mathematics and engineering}, Proc. National Acad. of Sciences,
{\bf 103} No. 7 (2006) 2032--2039.

\bibitem{CT1}  P.G. Casazza, M. Fickus, J.C. Tremain and E. Weber, {\em The Kadison-Singer Problem
in mathematics and engineering:  A detailed account},  Operator Theory,
IOperator Algebras and Applications, Proceedings of the 25th GPOTS Symposium (2005),
D. Han, P.E.T. Jorgensen and D.R. Larson Eds., Contemp. math.
{\bf 414} (2006) 299--355. 

\bibitem{Ch}  O. Christensen, {\em An introduction to frames and Riesz bases},
Birkhauser, Boston (2003).

\bibitem{T}  J.C. Tremain, {\em The Bourgain-Tzafriri Conjecture for operators with
small coefficients}, Preprint.

\end{thebibliography}
\end{document}